\pdfoutput=1
\documentclass{amsart}
\usepackage{latexsym}
\usepackage{amsfonts}
\usepackage{graphicx}

\newcommand{\R}{\mathbb{R}}
\newcommand{\Q}{\mathbb{Q}}

\newcommand{\To}{\rightarrow}
\newcommand{\vp}{\varphi}

\theoremstyle{plain}
\newtheorem{Thm}{Theorem}
\newtheorem{Cor}{Corollary}
\newtheorem{Lemma}[Cor]{Lemma}
\newtheorem{Prop}[Cor]{Proposition}

\newtheorem{Conj*}{Conjecture}

\theoremstyle{remark}

\newcommand{\J}{{J}acobian\ }
\newcommand{\p}{polynomial\ }
\newcommand{\pk}{{P}inchuk\ }
\newcommand{\nz}{nonzero\ }
\newcommand{\xyp}{$(x,y)$-plane\ }
\newcommand{\sa}{semi-algebraic\ }
\newcommand{\fe}{field extension\ }

\begin{document}
\title{{P}inchuk maps and  function fields}
\author{L. Andrew Campbell}
\address{908 Fire Dance Lane \\
Palm Desert CA 92211 \\ USA}
\email{landrewcampbell@earthlink.net}

\begin{abstract}
All  counterexamples of Pinchuk type  to the strong real {J}acobian 
conjecture are shown to have  rational function field 
extensions of degree six with no nontrivial automorphisms. 
\end{abstract}

\maketitle

\section{Introduction}\label{intro}

The unresolved 
\J Conjecture  asserts that 
a polynomial map $F: k^n \To k^n$, 
where $k$ is a field of characteristic zero, has a polynomial inverse if its \J determinant, $j(F)$, is a \nz element of $k$.
For $k=\R$ the similar Strong Real \J Conjecture was that 
$F$ is bijective if it is a \p map and 
 $j(F)$, though perhaps not a constant,
vanishes nowhere on $\R^n$. 
However, Sergey {P}inchuk exhibited a family 
of counterexamples for $n=2$, 
now usually called {P}inchuk maps.

Any \p endomorphism of $\R^n$, 
whose \J determinant does not vanish identically,
induces a finite algebraic extension of rational function fields,
and if it is bijective it is easy to show that the extension  is
of odd degree, with no nontrivial automorphisms.
This  extension  is investigated for a previously well studied {P}inchuk map. 
A primitive element is found, its minimal polynomial 
is calculated, and the degree ($6$) and automorphism group
(trivial) of the extension are determined. 
That generalizes to  any {P}inchuk map
 $F(x,y)=(P(x,y),Q(x,y))$ defined over
any subfield $k$ of $\R$. 
Although $F$ is generically two to one as a polynomial map of
$\R^2$ to $\R^2$, the degree of the associated 
extension of function fields  $k(P,Q) \subset k(x,y)$ 
is $6$ and $k(x,y)$ admits no nontrivial automorphism 
that fixes all the elements of $k(P,Q)$ (Theorem \ref{extension}).

\section{A specific Pinchuk map}\label{specific}

Pinchuk maps are certain polynomial maps $F=(P,Q): \R^2 \to \R^2$
that have an everywhere positive Jacobian determinant $j(P,Q)$,
and are not injective \cite{Pinchuk}. The polynomial $P(x,y)$
is constructed by defining
$t=xy-1,h=t(xt+1),f=(xt+1)^2(t^2+y),P=f+h$.
The polynomial $Q$ varies for different Pinchuk maps,
but always has the form
$Q =q - u(f,h)$, where $q= -t^2 -6th(h+1)$ and
$u$ is an auxiliary polynomial in $f$ and $h$,
chosen so that
$j(P,Q) = t^2 + (t+f(13+15h))^2 + f^2$.

The specific \pk map used here 
 is one introduced by Arno van den Essen via an email to colleagues in June 1994. It is defined
\cite{ArnoBook} by choosing
\begin{equation}\label{ueq}
u =
170fh + 91h^2 + 195fh^2 + 69h^3 + 75fh^3+
\frac{75}{4}h^4.
\end{equation}
The total degree in $x$ and $y$ of $P$ is $10$ and that of $Q$ is $25$. 

The image, multiplicity  and asymptotic behavior of $F$
were studied in
\cite{PPR,PPRErr,aspc}. Its asymptotic variety, $A(F)$,
is the set of points in the image plane that are finite limits of the value of $F$ along curves that tend to infinity  in the \xyp \cite{asympvals,asymptotics}. It may alternatively be defined as the set of points in the image plane that have no neighborhood with
a compact inverse image under $F$
\cite{notproper,realtrans,geometry}.
It is a topologically closed curve in the
image $(P,Q)$-plane and is the image of a real line under a bijective \p parametrization. 
It is depicted below using
differently scaled $P$ and $Q$ axes. It intersects the vertical axis at $(0,0)$ and $(0,208)$. Its leftmost point is $(-1,-163/4)$,
and that is the only singular point of the curve.

\begin{figure}[ht]
\centerline{
\includegraphics[height=2in,width=6in]{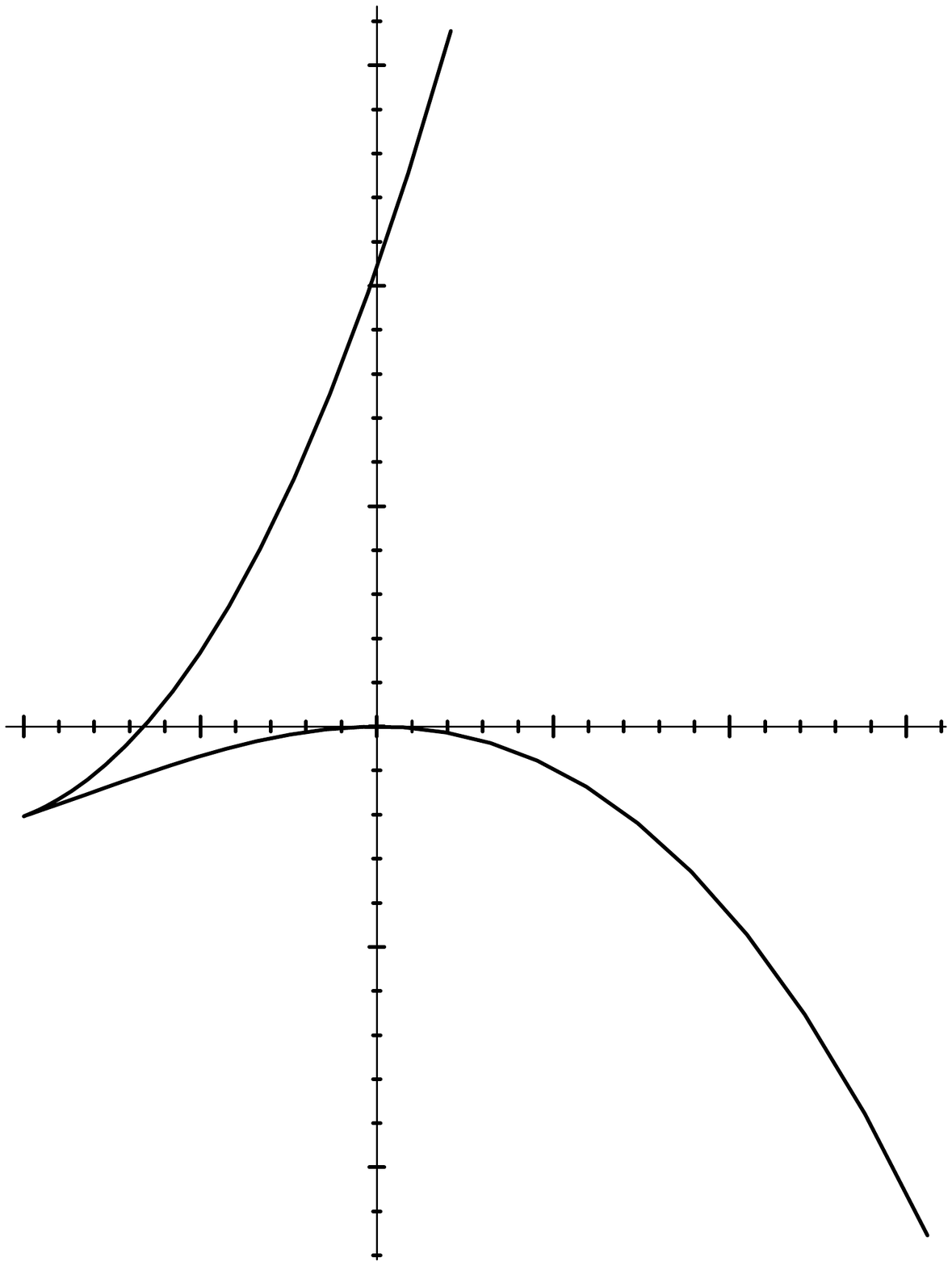}}
\caption{The asymptotic variety of the Pinchuk map $F$.}
\end{figure}

The image plane coordinates of points on $A(F)$ satisfy 
\cite{aspc} 
the irreducible \p equation
\begin{equation*}
(Q -(345/4)P^2 -231P -104)^2 = (P+1)^3(75P + 104)^2.
\end{equation*}
The point
$(-104/75,-18928/375)$ (approximately $(-1.38,-50.47)$)
also satisfies this equation, and so lies 
in the Zariski closure of $A(F)$, but does not lie on the curve $A(F)$ itself.
So $A(F)$ is an irreducible \sa set, 
but not an actual algebraic variety.

The points $(-1,-163/4)$ and $(0,0)$ of $A(F)$ have no inverse image under $F$, all other points of $A(F)$ have one inverse
image, and all points of the image plane not on $A(F)$ have two.

\section{Minimal polynomial calculation}\label{ext}

This paragraph is a summary of some key facts from previously cited work on $F$.
 A general level set $P=c$
in the $(x,y)$-plane has a rational parametrization.
Specifically, for any real $c$ that is not $-1$ or $0$,
the equations 
$$x(h) = \frac{ (c-h)(h+1) }{  (c-2h-h^2)^2 }$$
$$y(h) = \frac{ (c-2h-h^2)^2(c-h-h^2) }{ (c-h)^2 },$$
define a rational map pointwise on a real line with coordinate $h$, except where a pole occurs.
The use of $h$ as a parameter and the equality $P=c$ 
are consistent: on
 substitution 
the expression  $h(x(h),y(h))$ simplifies to  $h$,
and $P(x(h),y(h))$ to $c$.
There is always a pole at $h=c$ and
$Q(x(h),y(h))$ tends to $-\infty$ as the pole
is approached from either side. 
Also, $Q(x(h),y(h))$ tends to $+\infty$ as $h$ tends 
to $+\infty$ or
 to $-\infty$. 
If $c > -1$, there are two additional poles
at $h = -1 \pm \sqrt{1+c}$ and
$Q(x(h),y(h))$ tends to a finite asymptotic value at
each of these poles  as the pole
is approached from either side. 
For example, if $c=3$, there are poles at $-3$, $1$, and $3$,
and the corresponding values of $Q$ are 
$14965/4$, $-4235/4$, and $-\infty$, respectively. 
In such  cases
the asymptotic values are distinct and are the values of $Q$ at the two points of intersection of the vertical line $P=c$ and $A(F)$ in the $(P,Q)$-plane. 
The level sets $P=c$ are disjoint unions of 
their connected components, which are 
curves that are smooth (because of the \J condition) and tend 
to $\infty$ in the $(x,y)$-plane at both ends. The number of 
curves is two if $c < -1$, and four if $-1 < c \ne 0$. Even 
the two exceptional values fit this pattern, although they require 
different rational parametrizations, with $P=-1$ consisting of 
four curves and $P=0$ of five.

$F$ is not birational, because it is generically two to one.
Throughout this section, consider only the coefficient field $k=\R$.
To begin the exploration of the \fe $k(P,Q)\subset k(x,y)$,
rewrite the parametrization above in terms of $f$ and $h$,
using the relations  $P=c$ and $P=f+h$ to obtain
$$x =  f(h+1)(f-h-h^2)^{-2}$$
$$y = (f-h-h^2)^2(f-h^2)f^{-2},$$
which are identities in $k(x,y)$
(and so  $k(x,y) =  k(f,h)$).
It follows that
$xy = (h+1)(f-h^2)/f,
t  = xy -1 = [(h+1)(f-h^2)-f]/f
    = [fh-h^2-h^3]/f 
   = (h/f)[f-h(h+1)],
 q = -t^2-6th(h+1)
   = -h^2f^{-2}\{[f-h(h+1)]^2  + 6(h+1)[f-h(h+1)]f\}.$
In fact,
\begin{align*}
q &= -h^4(h+1)^2/f^2 + [2h^3(h+1)+6h^3(h+1)^2]/f
     + [-h^2 -6h^2(h+1)] \\
  &= -h^4(h+1)^2/f^2
   + h^3(h+1)(6h+8)/f
   -h^2(6h+7).
\end{align*}
Using that equation,  the definition $Q=q-u(f,h)$, and 
equation \eqref{ueq} for $u$, one can express $Q$ in terms 
of $f$ and $h$ alone. 
Clearing denominators $f^2Q = f^2q - f^2u$, or,
arranged by powers of $f$, 
\begin{align*}
f^2Q &=  -h^4(h+1)^2 \\
      &+f [h^3(h+1)(6h+8)] \\
      &+f^2[-h^2(6h+7)-91h^2 -69h^3-(75/4)h^4] \\
      &+ f^3[-170h-195h^2-75h^3].
\end{align*}
Now substitute $P-h$ for $f$ and collect in powers of $h$
to obtain a \p relation
\begin{equation*}
(197/4) h^6 + \cdots +(2PQ - 170P^3)h -P^2Q = 0.
\end{equation*}
Let $R(T)$  be the corresponding polynomial in $T$ with root $h$.  

Straightforward computations show that
\begin{align*}
R(T) &=(197/4)T^6 + (104 -(363/2)P)T^5
+(63-421P+(825/4)P^2)T^4\\
 &+(-306P+510P^2-75P^3)T^3
+(-Q+412P^2-195P^3)T^2\\
 &+(2PQ-170P^3)T -P^2Q.
\end{align*}

The coefficient of each power of $T$ is a \p in $P$ and $Q$ with rational  coefficients,
and has  total degree in $P$ and $Q$ at most $3$.
Since the leading coefficient of $R(T)$ is a real constant,
the fact that $R(h) = 0$  shows that $h$ is
integral over $k[P,Q]$. 
The identity $R(h) = 0$ in $\Q[x,y]$ was successfully
checked by a computer symbolic algebra program.

Let $m(T)$ be the polynomial in $k[P,Q][T]$, $ T$ an indeterminate,
which has leading coefficient $1$ and satisfies $ m(h) = 0$  in $k[x,y]$,
and which is of minimal degree. 
 Clearly $m$ is  irreducible in both 
$k[P,Q][T]$ and $k[P,Q,T]$, 
hence by the Gauss Lemma, in
$k(P,Q)[T]$. That implies that $m$ is also of minimal degree
over $k(P,Q)$, that $m$ divides any polynomial in $k[P,Q][T]$
with $h$ as a root, and that $m$ is unique.

Factorize $F$ as $\R^2 \To  V(m) \To \R^2$,
where the first map is $(P,Q,h)$, 
$V(m)$ is the zero locus of $m$ in $\R^3$, 
and the second map is the projection 
onto the first two components. 
The first map is birational ($k(P,Q)(h) = k(x,y)$) and the second
is finite  by integrality. If 
$w=(w_1,w_2)$  is a point of the $(P,Q)$-plane, call 
$r$ a root of $m$ over $w$, if 
$m(w_1,w_2,r)$ (write as $m(w,r)$) is zero.

\begin{Lemma}\label{gl}
For generic $w$, $m$ has exactly two real roots over $w$, they are simple and distinct, $w$ has exactly two inverse images $v$ and $v'$ under $F$, and the real roots of $m$ over $w$ are $r=h(v)$ and $r'=h(v')$.
\end{Lemma}

\begin{proof}
Take a bi-regular isomorphism from a Zariski open subset $O$ 
of the $(x,y)$-plane to a Zariski open subset of $V(m)$. 
The image is a nonsingular surface $S$.
In the usual (strong) topology it has
a finite number of connected components that are open subsets of $S$ and of $V(m)$. 
Take the union of i) the image of the complement of $O$ under $F$, ii) the projection of the complement of $S$, and iii)   $A(F)$.  The union is \sa  of maximum dimension $1$. Take $w$ in the complement of the Zariski closure of that union. Any root that lies over $w$ is a nonsingular point of  $V(m)$ (by construction), and so is a simple, not multiple, real root.
There are exactly two points, say $v$ and $v'$, that map to $w$ under $F$. Their images under $h$, $r$ and $r'$, lie over $w$.
By construction $(w,r)$ and $(w,r')$ lie in $S$, and $v$ and $v'$ lie in $O$. Hence $r$ and $r'$ are distinct. No point in the complement of $O$ can map to $(w,r)$ or $(w,r')$
(by construction), and $v$ and $v'$ are the only points
in $O$ that do so.
\end{proof}

\begin{Cor}
The $T$-degree of $m$ is even.
\end{Cor}

\begin{proof}
The complex roots over $w$ that are not real occur in
complex conjugate pairs.
\end{proof}

Let $m_0$ be the term of $m(T)$ of degree $0$ in $T$. Clearly, 
$m_0 \in k[P,Q]$ is not the zero polynomial. Since $m(T)$  
divides $R(T)$, $m_0$ divides $P^2Q$. Since the $T$-degree 
of m is even, $m_0(w)$ is the product of all the roots, real and complex, of $m$ over $w$, for any $w$ in the $(P,Q)$-plane. 

\begin{Prop}
$m_0$ is a positive constant multiple of $-P^2Q$.
\end{Prop}

\begin{proof}
It is known that $m_0$ divides $P^2Q$.
$F(1,0)=(0,-1)$ and $h(1,0)=0$, so $m_0(0,-1)=0$. 
This  shows that $P$ must divide $m_0$, for otherwise
$m_0(0,-1)$ would be nonzero. Next, $F(1,1)=(1,0)$ and $h(1,1)=0$, so $m_0(1,0)=0$. So $Q$ divides $m_0$.
Next, consider the union of
the vertical lines $P=c$ in the $(P,Q)$-plane, for $2 < c< 4$.
At least one such line must contain a point $w$ that is generic
in the sense of Lemma \ref{gl}, for otherwise there would be an
open set of nongeneric points. Choose such a $c$, and note that
all but finitely many points of the line $P=c$ are generic.
The level set $P=c$ in the $(x,y)$-plane has the rational parametrization by $h$ already described, with a pole at $h=c$, at which $Q$ tends to $-\infty$. 
Take a point $w = (c,d)$ with $d$ negative and sufficiently large.
Then $w$ will be generic and its two inverse images under $F$ will
have values of $h$ that approach $c$ as $d$ tends to $-\infty$,
one value of $h$ less than $c$ and the other greater,
The product of all the
roots of $m$ over $w$ will be positive, since the nonreal roots
occur as conjugate pairs. Since $P$ is positive and $Q$ negative,
the numerical coefficient of $m_0$ must be negative, regardless of
whether $m_0$ is exactly divisible by $P$ or by $P^2$.
Finally, make a
similar argument for a suitable line $P=c$, with $-4 < c < -2$.
Consideration of signs shows that $m_0$ must be divisible by $P^2$,
which yields the desired conclusion.
\end{proof}

\begin{Cor}\label{not2}
The $T$-degree of $m$ is not $2$.
\end{Cor}

\begin{proof}
If $m$ has degree $2$ and $w$ is generic, then $m_0(w)$ is exactly the product of the two real roots of $m$ over $w$.  But for the last two examples considered in the previous proof, the product tends to $c^2$ as $h$ tends to $c$, whereas $m_0(w)$ is unbounded.
\end{proof}

\begin{Prop}\label{R=cm}
$R(T) = (197/4)m(T)$.
\end{Prop}

\begin{proof}
As $m(T)$ is a nonconstant divisor of $R(T)$ of even $T$-degree not
equal to $2$, it remains only to show that the degree of
$m$ in $T$ is not $4$.
Suppose to the contrary that
$(m_4T^4+m_3T^3+m_2T^2+m_1T+m_0)
(d_2T^2+d_1T+d_0) = R(T)$,
where the first factor is $m(T)$, the second is a polynomial $D(T)$
 of degree $2$ in $T$,  and all the coefficients shown are in
$k[P,Q]$. 
Note that the equation is in $k[P,Q][T]=k[P,Q,T]$, 
where $P$, $Q$, and $T$ are simply algebraically independent 
variables. 
Equating leading and constant terms on both sides, one
finds that $m_4=1$, $d_2=197/4$ and $d_0$ is a positive constant.
The coefficient $d_1$ must also be constant. For if not,
$j = \deg^t(d_1) > 0$, where $\deg^t$ temporarily denotes the total degree in $P$ and $Q$.
As noted earlier, that $\deg^t$ is at most $3$
for every coefficient of $R(T)$.
Starting with $\deg^t(m_0)=3$ and equating in turn terms of
$T$-degree $1$ through $4$ on both sides of the equation
assumed for $R(T)$, one
readily finds that $\deg^t(m_4)=3+4j$. But $m_4=1$, a contradiction.
Thus $D(T)$ has constant coefficients. 
Next, set $P=0$ in $R(T)$,
obtaining $(197/4)T^6 + 104T^5 + 63T^4 - QT^2$. 
Further setting $Q=0$, one finds that the resulting \p in $T$ alone 
factors as $T^4((197/4)T^2 + 104T+ 63)$. Clearly $D$ must
be exactly the quadratic factor shown. But if $D(T)$  divides $R(T)$,
setting $P=0$ implies it must also divide $-QT^2$, which is absurd.
That contradiction shows that the original assumption to the contrary, that $m$ has $T$-degree $4$, is false.
\end{proof}

Recall that only the coefficient field $k=\R$ is considered 
in this section.

\begin{Cor}\label{six}
The \fe $\R(P,Q) \subset \R(x,y)$ is of degree $6$.
\end{Cor}

\begin{proof}
Clear. 
\end{proof}

\section{Automorphisms of the extension}\label{aut}

Again, assume 
throughout this section that $k=\R$.  
Let $Z=F^{-1}(A(F))$. 
For any $(x,y)\notin Z$ there is a unique different point 
$(x',y')\notin Z$ with the same image under $F$.  
That defines an involution $\tau$ of $\R^2 \setminus Z$,
that is a \sa real analytic diffeomorphism. 
Suppose $\varphi$ is an automorphism of 
$k(x,y)$ that is not the identity but 
fixes every element of $k(P,Q)$. 
If $\vp$ preserves $h$, then it also preserves $x$ and $y$, since 
they are rational functions of $P$ and $h$, namely 
\begin{equation}\label{ratexp}
x= \frac{ (P-h)(h+1) }{  (P-2h-h^2)^2 }
\text{  and  }
y= \frac{ (P-2h-h^2)^2(P-h-h^2) }{ (P-h)^2 }.  
\end{equation}
So $\vp(h)=h'\ne h$ and, furthermore, by the identity
principle for rational functions over $\R$, they cannot be equal on any nonempty open subset of the $(x,y)$-plane on which $h'$ is defined. 
Since
$\vp$ preserves both components of $F$, the fact that its geometric realization cannot be the identity even locally means that it must be $\tau$  wherever both are defined. 
That implies that there can be at most one such a nonidentity 
automorphism $\vp$. If it exists, then
the rational function $h'$ is analytic at any point 
$(x,y) \notin Z$, since $h'(x,y)=h(x',y')=h(\tau(x,y))$.

That reduces the question of the existence of $\vp$ 
to the following one. 
Is $h'$, a well defined 	real analytic function on the complement of $Z$ in the $(x,y)$-plane, in fact a real rational function?

\begin{Lemma}\label{existence}
There are three component curves of the level set $P=0$ on 
which $h$ is nonconstant and vanishes nowhere. 
On those curves 
$Q=Q(h)=h^2((197/4)h^2+104h+63)$ and is
everywhere positive. 
There is one point of $Z=F^{-1}(A(F))$ on the three curves. 
At all of their other points, $h'$ satisfies both 
$h' \ne h$ and $Q(h')=Q(h)$. 
\end{Lemma}

\begin{proof}
Set $P=0$ in equation \eqref{ratexp}. The resulting
 rational functions $x(h),y(h)$ are defined everywhere
except at $h=-2$ and $h=0$. 
That yields three curves parametrized by $h$.
Since $h(x(P,h),y(P,h))$ simplifies to $h$, 
$h(x(h),y(h))=h$ and hence $h$ assumes every real value 
exactly once on these curves, except that $-2$ and 
$0$ are never assumed. 
Since every level set $P=c$ is a finite disjoint union of 
closed connected smooth curves unbounded at both ends, each of 
the three curves is a connected component of $P=0$.

Set $P=0$ in the identity $R(h)=0$, obtaining
$(197/4)h^6+104h^5+63h^4-Qh^2=0$. 
On the curves, $h \ne 0$, so the claimed formula for 
$Q$ follows. 
Furthermore, $Q$ is positive there, 
since $(197/4)h^2+104h+63$ has negative 
discriminant.

So $F$ maps points on the three curves to the positive $Q$-axis. 
Routine calculation of the derivative of $Q$ shows that $Q$ is 
monotonic decreasing for $h<0$ and monotonic 
increasing for $h>0$. 
Considering the graph of $Q$, one concludes that every positive 
real is the value of $Q$ exactly twice, for \nz values of 
$h$ of opposite signs. Those values of $h$ all correspond to 
unique points of the curves, except $h=-2$. As $Q(-2)=208$, 
the point $(0,208)$, which is the only point of $A(F)$ on 
the positive $Q$-axis, has as its unique inverse under $F$ the 
point $(x(h'),y(h'))$, where $h'$ is the positive real 
satisfying $Q(h')=208$. 
\end{proof}

Remark. To clarify, 
there are two additional component curves of the level set 
$P=0$. On them $h=0$ identically. They have a rational 
parametrization by $t$ with a pole at $t=0$ and $Q=-t^2$ is
everywhere negative on them.

\begin{Lemma}\label{irrational}
Let $h'$ be any real rational function of $h$ that 
satisfies $Q(h')=Q(h)$ for infinitely many values 
of $h \in \R$. Then $h'=h$. 
\end{Lemma}

\begin{proof}
Suppose $h'=a/b$ for polynomials 
$a,b \in \R[h]$, of respective degrees $r,s$, with no common 
divisor. From $Q(h')=Q(h)$ one obtains 
\begin{equation}\label{abeq}
a^2((197/4)a^2+104ab+63b^2)=b^4h^2((197/4)h^2+104h+63),
\end{equation}
a polynomial equality that is true for all real $h$.  
Since $a/b$ tends to $\infty$ as $h$ does, 
$r>s$. Counting degrees $4r=4s+4$, so $r=s+1$. 
The factor in parentheses on the left is quadratic and 
homogeneous in 
$a$ and $b$ and has negative discriminant, so it is zero for real 
$a$ and $b$ only if both are zero. But that cannot occur for 
any real $h$, for then $a$ and $b$ would have a common root, 
hence a common divisor. 
Thus that factor has only complex roots. It follows that 
$h^2$ divides $a^2$, and so $h$ divides $a$. 
That means that $a$ and $bh$ share a root, each having 
$h$ as a factor. As the quadratic factor in parentheses on the 
right is not zero for any real $h$, any further real roots 
(at $h=0$ or not) in the two sides of equation \eqref{abeq}
would be shared by $a$ and $b$. Again, that is not possible, 
and therefore $a$ has no further real roots and all the roots of 
$b$ are complex. In particular $s$, the degree of $b$, must 
be even. No complex root of $b$ can be a root of $a$, as that 
would imply a common real irreducible quadratic factor. 
So it must be a root of the parenthetical factor 
on the left. Counting roots with multiplicities, 
$2r=2s+2 \ge 4s$, so $s \le 1$. Since $s$ is even, it 
must be $0$, and so $h'=\lambda h$ for a 
\nz $\lambda \in \R$. 
Then for any fixed $h \ne 0$, $Q(\lambda^ih)=Q(h)$ is 
independent of $i>0$, and so $\lambda$ has absolute value 
$1$. It cannot be $-1$, because $Q(h)-Q(-h)=208h^3$. 
\end{proof}

\begin{Prop}\label{notg}
The group of automorphisms 
of the  \fe $\R(P,Q) \subset \R(x,y)$ is trivial. 
\end{Prop}

\begin{proof}
If the group contains a nontrivial automorphism $\vp$, 
then $h'=\vp(h)=h(x',y')$ (see above) belongs to
$\R(x,y)=\R(P,h)$. As $\R(P,h)$ is a rational function field 
in two algebraically independent elements over $\R$, the 
restriction of $h'$ to the level set $P=0$ must either be 
generically undefined (uncanceled $P$ in the denominator) or 
a rational function of $h$. 
The first case is ruled out by Lemma \ref{existence}, which 
also contradicts Lemma \ref{irrational} in the second case.
\end{proof}

\section{All {P}inchuk maps}\label{main}

From a geometric point of view,
any two different \pk maps are very closely related.
More specifically, if $F_1=(P,Q_1)$ and $F_2=(P,Q_2)$
are \pk maps then they have the same first component, $P$,
and their second components satisfy $Q_2=Q_1+S(P)$
for a \p $S$ in one variable with real 
coefficients \cite{aspc}.
As maps of $\R^2$ to $\R^2$, therefore, they differ only by
a triangular \p automorphism of the image plane. 
So  all \pk maps are generically two to one,
 and their asymptotic varieties
have algebraically isomorphic embeddings in the image plane.

Let $F$ be the same \pk map as before.  It is defined over
$\Q$. In fact, not only do $P$ and $Q$ have rational coefficients,
but so do $h$ and all terms of the minimal \p $m$ for $h$.
Let $k$ be any  subfield of $\R$, including the possibilities
$k=\Q$ and $k=\R$. Then the powers 
$h^i$ for $i=0,\ldots,5$
form a basis for $k(x,y)$ as a vector space over $k(P,Q)$ 
and the \fe is of degree $6$.

\begin{Prop}
Let $F_1$ and $F_2$ be \pk maps defined over $k$ 
and connected by $S$. 
Then $S$ is uniquely determined and its coefficients
belong to $k$. 
\end{Prop}

\begin{proof}
$P$ is transcendental over $\R$, so $S$ is unique. 
Let $c \in \Q$ with $c \ne 0$ and $c \ne -1$.
Choose $h \in \Q$ that is not a pole of the  previously described rational parametrization $x(h),y(h)$ of the level set $P=c$.
Since both $x(h)$ and $y(h)$ have formulas 
in $\Q(h,c)$, the real number
$S(c)=Q_2(x(h),y(h))-Q_1(x(h),y(h))$ actually is in $k$.
The coefficients of $S$ can be reconstructed, using rational arithmetic, from its values at any $j > \deg S$ such points $c$,
and so are in $k$.
\end{proof}

\begin{Cor}
Any two \pk maps defined over $k$ 
have one and the same field extension over $k$.
\end{Cor}

\begin{Thm}\label{extension}
Let $F$ be any \pk map and let $k$
be $\R$ or any subfield of $\R$
containing  the coefficients of $F$.
Although $F$ is generically two to one as a \p map of
$\R^2$ to $\R^2$, the degree of the associated 
extension of function fields over $k$ is six. 
Furthermore, the extension has no automorphisms 
other than the identity. 
\end{Thm}

\begin{proof}
The conclusions have already been drawn for the earlier  specific 
\pk map $F=(P,Q)$ and for 
$k=\R$ (Corollary \ref{six} and Proposition \ref{notg}). 
Both \pk maps have the same function \fe over $k$, so 
it has degree six. And any nontrivial automorphism 
defined over $k$ defines one over $\R$, 
since the $\R$-linearly extended 
automorphism  preserves $P$, $Q$, and $\R$. 
\end{proof}

\end{document}